\documentclass[12pt,reqno]{amsart}

\usepackage{amssymb, amsmath, amsfonts, latexsym}
\usepackage{enumerate}
\usepackage{color}
\newtheorem{thm}{Theorem}[]
\newtheorem{lem}{Lemma}[section]

\theoremstyle{definition}

\numberwithin{equation}{section} \theoremstyle{remark}

\title[Convergence rate for Ginibre]{\bf  Exact convergence rate of spectral radius of complex Ginibre to Gumbel distribution}

\author{Y{utao} Ma}
\address{Yutao MA\\ School of Mathematical Sciences $\&$ Laboratory  of Mathematics and Complex Systems of Ministry of Education, Beijing Normal University, 100875 Beijing, China.} 
\thanks{The research of Yutao Ma was supported in part by NSFC 12171038 and 985 Projects.}
\email{mayt@bnu.edu.cn}

\author{Xujia Meng}
\address{Xujia Meng\\ School of Mathematical Sciences $\&$ Laboratory of Mathematics and Complex Systems of Ministry of Education, Beijing Normal University, 100875 Beijing, China.}
\email{202321130122@mail.bnu.edu.cn}

\begin{document}
\maketitle

\begin{abstract}

Consider the complex Ginibre ensemble, whose eigenvalues are $(\lambda_i)_{1\le i\le n}$ and the  
spectral radius $R_n=\max_{1\le i\le n}|\lambda_i|.$ Set $X_n=\sqrt{4 \gamma_{n}}(R_{n}-\sqrt{n}-\frac12\sqrt{\gamma_{n}})$ and $F_n$ be its distribution function, where 
$\gamma_{n}=\log n-2\log(\sqrt{2\pi}\log n).$ It was proved in \cite{Rider 2003} that $F_n$ converges weakly to the Gumbel distribution $\Lambda.$ We prove in further in this paper  that 
$$\lim_{n\to\infty} \frac{\log n}{\log\log n}\, W_1\left(F_n, \Lambda\right)=2$$
   and the Berry-Esseen bound 
      $$\lim\limits_{n\to \infty} \frac{\log n}{\log\log n}\sup_{x\in \mathbb{R}}|F_{n}(x)-e^{-e^{-x}}|=\frac{2}{e}.$$
	\end{abstract} 


{\bf AMS Classification Subjects 2020:} 60F10, 60B20, 60G55

\section{Introduction}	

Random Matrix Theory (RMT) originated from the study of the energy levels of large particles in quantum mechanics \cite{Miodrag 2011}. Quantum Hamiltonians are naturally self-adjoint, which leads to traditional random matrix models being Hermitian. As a result, early research  concentrated on Hermitian random matrices, like the Wigner matrix. In 1965, motivated by mathematical interest, Ginibre broadened this scope by exploring non-Hermitian Gaussian ensembles with real, complex, or quaternionic entries \cite{ginibre}.
		
	The Ginibre ensemble (\cite{ginibre}) is the simplest and most commonly used prototype of non-Hermitian random matrices, which consists of $n\times n$ random matrices $X$ whose entries are independent, identically distributed (i.i.d.) standard real Gaussian or complex Gaussian. The Ginibre ensemble has numerous practical applications across different domains. In statistical physics, it has been  utilized to investigate diffusion processes, persistence, and the equilibrium counting of random nonlinear differential equations, which help in analyzing the stability of complex systems \cite{Byun 2025}. In quantum physics \cite{Byun 2023}, the Ginibre ensemble is employed to describe topologically driven parameter level crossings in certain quantum dots.  Furthermore, it has been employed to examine financial markets, as demonstrated in reference \cite{Biely 2008}. 	
	
	A significant finding related to the Ginibre ensemble is that the empirical measures of eigenvalues converges to the Girko's circular law, whose convergence rate with respect to $W_1$ distance was obtained in \cite{Meckes 2015}, and proven in generality without Gaussian assumption \cite{Bai97, Mehta 1967, Girko, Tao-Vu}. The reader also can find the (mesoscopic) central limit theorem for linear eigenvalue statistics related to the real or complex Ginibre ensembles in \cite{BCH24, Cipo23, Forrester99, RiderSil06, RiderVirag07, TaoVu15}  etc.
	Meanwhile, the spectral radius converges to $1$ \cite{9, 11, 12, 20} with an explicit speed \cite{7}. 
		
Let $G_n$ be an $n\times n$ random matrix with i.i.d. standard complex Gaussian entries and let $\{\lambda_i\}_{1\le i\le n}$ be its eigenvalues.  Let $R_n$ be the spectral radius of $G_n,$ which is defined as $ R_{n}=\max\limits_{1\leq k\leq n}|\lambda_{k}|.$ For convenience, denote 
$$X_n:=\sqrt{4 \gamma_{n}}\big(R_{n}-\sqrt{n}-\sqrt{\frac{\gamma_{n}}{4}}\big),$$
	where $\gamma_{n}=\log n-2\log(\sqrt{2\pi}\log n).$
	Rider  \cite{Rider 2003} investigated the weak convergence of $R_n,$ which says 
		$$\lim\limits_{n\to \infty} \mathbb{P}\left(X_{n}\leq x\right)=e^{-e^{-x}},$$
where  $ e^{-e^{-x}}$ is the  distribution function of the Gumbel distribution $ \Lambda $.
Lacking radial symmetry, which is a key element of Kostlan's observation, the analogous result for the real Ginibre ensemble required a much more sophisticated analysis by Rider and Sinclair in \cite{Rider 14}. They demonstrated that the limiting distribution of the largest real eigenvalue and the largest imaginary eigenvalue of the real Ginibre ensemble converge, and based on this, they proved that in the case of the real Ginibre ensemble, $X_{n}$
converges to a slightly rescaled Gumbel distribution with distribution function   $ e^{-\frac{1}{2}e^{-x}}.$ Similar results for the largest real part of the eigenvalues were obtained for real and complex Ginibre ensembles in \cite{Akemann 2014} and \cite{Bender 2010}, respectively and	
the latest work \cite{Cipolloni 2022} provides the convergence rate of order $\frac{(\log\log n)^{2}}{\log n}$ for the largest real part of real and complex cases. 

The Gumbel distribution has universality in large amount of random matrices with i.i.d. entries. For example, let $M_n$ be a  complex random matrix with iid entries $x_{a b}\stackrel{d}{=}n^{-1/2}\chi$ and $\chi$ satisfies 
$$\mathbb{E}\chi=0, \quad \mathbb{E}|\chi|^2=1, \quad \mathbb{E}\chi^2=0, \quad \mathbb{E}|\chi^p|\le c_p$$ for any $p\in \mathbb{N}.$  
Recently, Cipolloni {\it et al.} \cite{CipoErdoXu23} proved that the spectral radius, with the same scaling as complex Ginibre ensemble, converges weakly to the Gumbel and similar asymptotic holds for real case, which solves the long-standing conjecture by Bordenave and Chafa$\ddot{\rm \imath}$ \cite{28}, \cite{42}(see also \cite{43}, \cite{Rider 14}).    

A natural and important subsequent object is to find out the speed of convergence to the Gumbel distribution. In this paper, thanks to Kostlan's observation again, we are able to give the exact convergence rate for the spectral radius of the complex Ginibre ensemble. 

We first introduce the $W_1$-distance on $\mathbb{R}$. 
The $W_{1} $ distance between probability measures  $ \mu $ and $ \nu $ on $ \mathbb{R}$ (see \cite{Villani}) is defined as
	$$W_{1}(\mu,\nu)=\mathop{\inf\limits_{X\sim\mu}}\limits_{Y\sim\nu}\mathbb{E}d(X, Y), $$
	where the infimum is taken over all couplings of random variables	$ X $ and 	$ Y  $ with marginal distributions 
	$ \mu $ and $ \nu $, respectively.  For this particular case, there is a closed formula for $W_1$ saying 	
\begin{equation}\label{W1}
	W_{1}(\mu, \nu)=\int_{\mathbb{R}}| F_{\mu}(t)-F_{\nu}(t)|dt,
\end{equation}
where $F_{\mu}$ and $F_{\nu}$ 
are the  distribution functions of 
$ \mu $ and 
$ \nu  $, respectively. 
	
We primarily investigate the convergence rate of the spectral radius $R_n$ with respect to $W_1.$ 
\begin{thm}\label{TH}
	Let 
	$ G_{n} $ 
		be a complex Ginibre ensemble and $R_n$ be its spectral radius.  
	Let $F_n$ be the distribution function of $X_{n}=\sqrt{4 \gamma_{n}}\left(R_{n}-\sqrt{n}-\sqrt{\frac{\gamma_{n}}{4}}\right)$ with $\gamma_{n}=\log n-2\log(\sqrt{2\pi}\log n).$	
	Then 
	$$ \lim_{n\to \infty} \frac{\log n}{\log \log n} W_{1}\left(F_{n},\Lambda\right)=2.$$
	
\end{thm}

Next, we state the Berry-Esseen bound between $F_n$ and $\Lambda.$
\begin{thm}\label{th2}
Let $X_{n}$ be the same as above and $F_{n}$ be its distribution function.	
 Then
	$$\lim\limits_{n\to \infty} \frac{\log n}{\log \log n}\sup_{x\in \mathbb{R}}|F_{n}(x)-e^{-e^{-x}}|=\frac{2}{e}.$$
	\end{thm}

The remainder of this paper will be organized as follows: the second section is concentrated on preparations and the third section is devoted to the proof of Theorem \ref{TH} and we give a brief proof of Theorem \ref{th2} in the last section. 

Hereafter, we use frequently $t_n=O(z_n)$ or $t_n=o(z_n)$ if $\lim_{n\to\infty}\frac{t_n}{z_n}=c\neq 0$  or $\lim_{n\to\infty}\frac{t_n}{z_n}=0,$ respectively and $t_n=\tilde{O}(z_n)$ if $\lim_{n\to\infty}\frac{t_n}{z_n}$ exists.  We also use $t_n\ll z_n$ or equivalently $z_n\gg t_n$ to represent $\lim_{n\to\infty}\frac{t_n}{z_n}=0.$

\section{Preparation work} 

In this section, we do some preparations for the proof of Theorem \ref{TH}. 

Let $ G_{n} $ be a complex
	Ginibre ensemble and $(\lambda_1, \cdots, \lambda_n)$ be its eigenvalues.  The exact joint density function of the eigenvalues of the  matrix $ G_{n} $ was derived by Ginibre \cite{ginibre}   
	\begin{equation}\label{pn}
		f_{n}(z_{1},\dots,z_{n})=\dfrac{1}{Z_{n}}e^{-\sum_{k=1}^{n}|z_{k}|^{2}}\prod\limits_{1\leq j<k\leq n}|z_{j}-z_{k}|^{2}, 
	\end{equation} 
	where $ Z_{n} $ is the normalizer.  Recall that
 $ R_{n}=\max\limits_{1\leq k\leq n}|\lambda_{k}|. $

The first crucial lemma is based on Corollary 1.2 in \cite{Kostlan}, which transfers the corresponding probabilities of 
 $R_n$ to those related to independently random variables taking values in real line. 
    
\begin{lem}\label{RY}
	Let $ \{Y_{j}\}_{j\ge 1}$ be a sequence of independent random variables, whose density function is  proportional to 
	 $y^{j-1}e^{-y}1_{y>0}$ and $Y_{(n)}:=\max\limits_{1\leq k\leq n}Y_{k}.$ Let $R_{n}$ be defined as above.  Then, $R_{n}^{2}$ has the same distritution as  $Y_{(n)}.$
\end{lem}

We first recall our target 
$$X_n:=\sqrt{4 \gamma_{n}}\big(R_{n}-\sqrt{n}-\frac12\sqrt{\gamma_{n}}\big)$$ 
with $F_n$ being its distribution function and then we write $X_n$ as a function of $R_n^2$ 
\begin{equation}\label{Xn2}X_n=\frac{\sqrt{4 \gamma_{n}}\big(R^{2}_{n}-(n+\sqrt{n\gamma_n}+\frac{1}{4}\gamma_{n})\big)}{R_{n}+\sqrt{n}+\frac12\sqrt{\gamma_{n}}}.\end{equation}  
The fact $X_n$ converging weakly to $\Lambda$ implies that $\frac{R_n}{\sqrt{n}}\to 1$ weakly, which is then equivalent to say 
$$W_n:=\frac{\sqrt{\gamma_{n}}\big(R^{2}_{n}-(n+\sqrt{n\gamma_n})\big)}{\sqrt{n}}\to \Lambda \quad \text{weakly}$$
as $n\to\infty.$ Here, the denominator in \eqref{Xn2} is replaced by $2\sqrt{n}$ because $\frac{R_n}{\sqrt{n}}\to 1$ weakly and the term $\frac{1}{4}\gamma_{n}$ disappears in the numerator since $\gamma_n=o(\sqrt{n}).$

Let $\mathcal{L}(W_n)$ be the distribution of $W_n.$ For the target $W_1(F_n, \Lambda),$ we first study $W_1(\mathcal{L}(W_n), \Lambda),$ which is easier to be figured out. Setting for simplicity $$ a_{n}:=n+\sqrt{n\gamma_{n}} , \quad b_{n}:=\sqrt{\frac{n}{\gamma_{n}}},$$ 
one gets 
$$\mathbb{P}(W_n\ge x)=\mathbb{P}(R_n^2\ge a_n+b_n x)=\mathbb{P}(Y_{(n)}\ge a_n+b_n x).$$ 
Hence, we next work on  
$\mathbb{P}(Y_{(n)}\ge a_n+b_n x).$ 
\begin{lem}\label{ZX1}
	Given $0\leq k\ll n $ and set 
	$$u_{n}(k,x):=\frac{k}{\sqrt{n}}+\sqrt{\gamma_{n}}+\frac{x}{\sqrt{\gamma_{n}}} .$$ Then $$
		\mathbb{P}(Y_{n-k}> a_{n}+b_{n}x)\le\frac{1}{u_{n}(k,x)}e^{-\frac{u_{n}^{2}(k,x)}{2}} +O(n^{-1/2}u_{n}^{-3}(k,x))
	$$
	uniformly on $k $ and  $ x $ such that $ u_{n}(k,x)\gg 1 $. 	Furthermore, when $ 1\ll u_{n}(k,x)=\tilde{O}(n^{\frac{1}{10}}),$ we have a precise asymptotic
	$$\aligned 
		\mathbb{P}(Y_{n-k}>a_{n}+b_{n}x)=\frac{1+O(u_n^{-2}(k, x))}{\sqrt{2\pi} u_n(k, x)}e^{-\frac{u_n^2(k, x)}{2}}. 
	\endaligned $$
\end{lem}
\begin{proof} The density function of $Y_j$ is $\frac{1}{(j-1)!}y^{j-1} e^{-y}, y>0.$ Equivalently, $Y_j$ is a Gamma distribution and it could be regarded in distribution as the sum of $j$ identically independent random variables obeying exponential distribution with parameter $1.$ Thus, Theorem 1 on page 293 from \cite{Petrov}  entails that
$$\aligned \mathbb{P}(Y_{n-k}>a_{n}+b_{n}x)&=\mathbb{P}\left(\frac{Y_{n-k}-(n-k)}{\sqrt{n-k}}>\frac{a_{n}+b_{n}x-(n-k)}{\sqrt{n-k}}\right)\\
&=(1-\Phi(\frac{a_{n}+b_{n}x-(n-k)}{\sqrt{n-k}}))(1+O(n^{-1/2}u_n^{3}(k, x)))  \endaligned $$	
once $|\frac{a_{n}+b_{n}x-(n-k)}{\sqrt{n-k}}|\ll n^{1/6},$
where $ \Phi $ is the distribution function of a standard normal. We simplify first $\frac{a_{n}+b_{n}x-(n-k)}{\sqrt{n-k}}.$ Indeed, with the condition $ 0\leq k\ll n $, we have 
	\begin{equation}\aligned \label{an}
		\frac{a_{n}+b_{n}x-(n-k)}{\sqrt{n-k}}&=\frac{k+\sqrt{n\gamma_{n}}+\sqrt{\frac{n}{\gamma_{n}}}\,x}{\sqrt{n}}\sqrt{\frac{n}{n-k}}\\
		&=u_n(k, x)\left(1+O\left(kn^{-1}\right)\right). \endaligned 
	\end{equation}
	The condition $1\ll u_n(k, x)=\tilde{O}(n^{\frac{1}{10}})$ indicates $k=\tilde{O}(n^{3/5})$ and then $u_n(k, x) k n^{-1}=o(1).$ 
	Thus, the Mills ratio \begin{equation}\label{normal}1-\Phi(t)=\frac{1}{\sqrt{2\pi} t} e^{-t^2/2}(1+O(t^{-2}))\end{equation}
for $t>0$ large enough  helps us to get in further that 
$$\aligned &\quad 
\mathbb{P}(Y_{n-k}> a_{n}+b_{n}x)\\
&=\frac{1+O(n^{-1/2}u_n^3(k, x)+kn^{-1}+u_n^{-2}(k, x))}{\sqrt{2\pi} u_n(k, x)}e^{-\frac{u_n^2(k, x)}{2} (1+O(k n^{-1}))}\\
&=\frac{e^{-\frac{u_n^2(k, x)}{2}}}{\sqrt{2\pi} u_n(k, x)}(1+O(n^{-1/2}u_n^3(k, x)+u_n^{-2}(k, x)+u_n^2(k, x)kn^{-1})). 
\endaligned $$ 
Since $u_n(k, x)=\tilde{O}(n^{\frac{1}{10}}),$ we see clearly 
$$O(n^{-1/2}u_n^3(k, x)+u_n^{-2}(k, x)+u_n^2(k, x)kn^{-1})=O(u_n^{-2}(k, x)).$$
Now, we work for the upper bound. We see from \eqref{an} that 
$$\frac{a_{n}+b_{n}x-(n-k)}{\sqrt{n-k}}\ge u_n(k, x),$$
whence with the help of the non-uniform Berry-Esseen bounds ((8.1) in \cite{Chen 2011}) and the Mills ratio again, we have  
$$	\aligned
	\mathbb{P}(Y_{n-k}> a_{n}+b_{n}x)&\le \mathbb{P}\left(\frac{Y_{n-k}-(n-k)}{\sqrt{n-k}}\ge u_n(k, x)\right)\nonumber\\
	&=1-\Phi(u_n(k, x))+O(n^{-1/2}u_n^{-3}(k, x)), \nonumber\\
	&=\frac{1+o(1)}{\sqrt{2\pi}u_n(k, x)} e^{-\frac{u_n^2(k, x)}{2}}+O(n^{-1/2}u_n^{-3}(k, x))\nonumber\\
	&\le \frac{1}{u_n(k, x)}e^{-\frac{u_n^2(k, x)}{2}}+O(n^{-1/2}u_n^{-3}(k, x))\endaligned $$
uniformly for $0\leq k\ll n $ and  $ x $ such that $1\ll  u_{n}(k,x)$ as $n$ large enough. 
\end{proof}

Next, we summarize some formulas relating to $ u_{n}(k,x) $ in one lemma.
\begin{lem}\label{3.1}
Given $0\le k\le j_n$ with $\sqrt{n\gamma_n}\ll j_n\ll n.$ Let $u_n(k, x)$ and $\gamma_n$ be defined as above. Given any $ x_{n}$ such that
$ 1\ll u_{n}(k, x_{n}) $ for all $0\le k\le j_n.$
	
\item[(1).] For any $m>1,$ it follows 
$$\int_{x_n}^{\infty}u^{-m}_n(k, x)dx=\frac{\sqrt{\gamma_{n}}}{m-1} u^{-m+1}_n(k, x_n).$$ 	
\item[(2).] For any $c, m>0,$ we have 
$$\int_{x_n}^{\infty} u^{-m}_n(k, x)e^{-c \,u^2_n(k, x)}dx=\frac{\sqrt{\gamma_{n}} e^{-c u_n^2(k, x_n)}}{2cu_n^{m+1}(k, x_n)} (1-\frac{m+1}{2c} u_n^{-2}(k, x_n)+ O(u_n^{-4}(k, x_n))).$$	
\item[(3).] Given $ L\geq 0, c, m>0 $ and  suppose $1\ll u_{n}(L, x_{n})\ll \sqrt{n}.  $ Then
$$\aligned 
	&\quad \sum\limits_{k=L}^{+\infty}u_{n}^{-m}(k,x_{n})e^{-c u_{n}^{2}(k,x_{n})}\\&=\frac{\sqrt{n}e^{-c u_{n}^{2}(L, x_{n})}}{2 c u_{n}^{m+1}(L, x_{n})}(1-\frac{m+1}{2c} u_n^{-2}(L, x_n)+ O(u_n^{-4}(L, x_n)+u_n(L, x_n)n^{-1/2})).\endaligned 
$$	
	\item[(4).] Given $m>1$ and $ L\geq 0.$  Assuming $1\ll u_{n}(L, x_{n}),$  one has  
	$$\sum_{k=L}^{+\infty} u_n^{-m}(k, x_n)=\frac{\sqrt{n}}{m-1}u_n^{-m+1}(L, x_n)(1+O(u_n^{-1}(L, x_n) n^{-1/2})).$$
\end{lem}
\begin{proof}
	Using the substitution $t=u_n(k, x),$ which satisfies $dt =\frac{dx}{\sqrt{\gamma_{n}}},$ we see 
	$$\int_{x_n}^{\infty}u^{-m}_n(k, x)dx=\sqrt{\gamma_{n}}\int_{u_n(k, x_n)}^{\infty}t^{-m} dt=\frac{\sqrt{\gamma_{n}}}{m-1} u^{-m+1}_n(k, x_n).$$	
	Similarly, it holds 
	$$\aligned \int_{x_n}^{\infty} u^{-m}_n(k, x) e^{-c \,u^2_n(k, x)}dx&=\sqrt{\gamma_{n}} \int_{u_n(k, x_n)}^{\infty} t^{-m}e^{-c \,t^2} dt.
	\endaligned $$
	 For the integral $\int_{z}^{\infty} t^{-m} e^{-ct^2}dt$ with $z$ large enough, 
	we have via the substitution $y=t^2$ and twice integral formula by parts that 
	\begin{equation}\label{oneasym}\aligned \int_{z}^{\infty} t^{-m} e^{-ct^2}dt&=\frac{1}{2}\int_{z^2}^{\infty} y^{-\frac{m+1}{2}}e^{-cy} dy\\
	&=\frac{1}{2c} e^{-cz^2} z^{-(m+1)}\big(1-\frac{(m+1)}{2c}z^{-2}+O(z^{-4})\big).
	\endaligned  
	\end{equation} 
This expression for the integral immediately tells 
$$\aligned \int_{x_n}^{\infty} u_n^{-m}(k, x) e^{-c \,u^2_n(k, x)}dx&=\frac{\sqrt{\gamma_{n}} e^{-c u_n^2(k, x_n)}}{2cu_n^{m+1}(k, x_n)} (1-\frac{m+1}{2c} u_n^{-2}(k, x_n)+ O(u_n^{-4}(k, x_n))).
\endaligned $$
For the third term, we extend the definition of	 $u_n(k, x)$ to 
$u_n(\cdot, x)$ on $\mathbb{R}_+$ as 
$$u_n(t, x)=\frac{t}{\sqrt{n}}+\sqrt{\gamma_{n}}+\frac{x}{\sqrt{\gamma_{n}}} .$$ 
	Since $ u_{n}^{-m}(k,x_{n})e^{-cu_{n}^{2}(k,x_{n})} $ is decreasing on $ k $, it is ready to see $\sum_{k=L}^{+\infty}u_{n}^{-m}(k,x_{n})e^{-cu_{n}^{2}(k,x_{n})}$ lies in the interval 
\begin{align}
	\left(\int_{L}^{+\infty}u_{n}^{-m}(t,x_{n})e^{-cu_{n}^{2}(t,x_{n})}dt,  \;
	\int_{L}^{+\infty}u_{n}^{-m}(t, x_{n})e^{-cu_{n}^{2}(t, x_{n})}dt+u_{n}^{-m}(L,x_{n})e^{-cu_{n}^{2}(L,x_{n})}\right). \label{interval}
\end{align}
	The substitution $y=u_n(t, x_n),$ which verifies $d y=\frac{dt}{\sqrt{n}},$ and \eqref{oneasym} imply that 
\begin{align*}
	\int_{L}^{+\infty}u_{n}^{-m}(t,x_{n})e^{-cu_{n}^{2}(t,x_{n})}dt&=\sqrt{n}\int_{u_{n}(L, x_{n})}^{+\infty}y^{-m} e^{-c y^{2}}dy\\
	&=\frac{\sqrt{n}e^{-c u_{n}^{2}(L, x_{n})}}{2 c u_{n}^{m+1}(L, x_{n})}(1-\frac{m+1}{2c} u_n^{-2}(L, x_n)+ O(u_n^{-4}(L, x_n))).\nonumber
\end{align*}
Under the condition $u_{n}(L, x_{n})=o(\sqrt{n}), $ we see 
$$\frac{u_{n}^{-m}(L,x_{n})e^{-cu_{n}^{2}(L,x_{n})}}{\int_{L}^{+\infty}u_{n}^{-m}(t, x_{n})e^{-cu_{n}^{2}(t, x_{n})}dt}=O\left(\frac{u_{n}(L, x_{n})}{\sqrt{n}}\right)=o(1).$$ 
Therefore, the interval \eqref{interval} ensures that 
$$\aligned 	& \quad \sum_{k=L}^{+\infty}u_{n}^{-m}(k, x_{n})e^{-cu_{n}^{2}(k, x_{n})}\\
&=\frac{\sqrt{n}e^{-c u_{n}^{2}(L, x_{n})}}{2 c u_{n}^{m+1}(L, x_{n})}\left(1-\frac{m+1}{2c} u_n^{-2}(L, x_n)+O(u_n^{-4}(L, x_n)+u_{n}(L, x_{n})n^{-1/2})\right). \endaligned$$
Thus, we finish the third item. The fourth item follows by the same method as that
for the third one.
\end{proof}

\section{Proof of Theorem \ref{TH}}
This section is devoted to the proof of Theorem \ref{TH}. 

Recall $$X_n:=\sqrt{4 \gamma_{n}}\big(R_{n}-\sqrt{n}-\frac12\sqrt{\gamma_{n}}\big)$$ 
and 
$$W_n:=\frac{\sqrt{\gamma_{n}}\big(R_{n}^2-(n+\sqrt{n\gamma_n})\big)}{\sqrt{n}}.$$
We first give a sketch of the proof. 

The first thing to do is to transfer 
$W_1(F_n, \Lambda)$ to $W_1(\mathcal{L}(W_n), \Lambda).$
In fact, the triangle inequality of $W_1$ says
$$W_1(\mathcal{L}(W_n), \Lambda)-W_1(\mathcal{L}(W_n), F_n)\le W_1(F_n, \Lambda)\le W_1(\mathcal{L}(W_n), \Lambda)+W_1(\mathcal{L}(W_n), F_n).$$ 
To prove Theorem \ref{TH}, we will prove that 
\begin{equation}\label{total0} W_1(\mathcal{L}(W_n), \Lambda)=2(1+o(1))\gamma_n^{-1} \log(\sqrt{2\pi}\log n) \end{equation}
 and 
 \begin{equation}\label{total1} W_1(\mathcal{L}(W_n), F_n)\ll \gamma_n^{-1}.\end{equation}
Now the closed form \eqref{W1} of $W_1$ says 
$$W_{1}(\mathcal{L}(W_n),\Lambda)=\int_{-\infty}^{+\infty}|\mathbb{P}(W_n\le x)-e^{-e^{-x}}|dx=\int_{-\infty}^{+\infty}|\mathbb{P}(Y_{(n)}\le a_n+b_n x)-e^{-e^{-x}}|dx,$$ whence 
\begin{align}
	W_{1}(\mathcal{L}(W_n),\Lambda)&=\left(\int_{-\infty}^{-\ell_1(n)}+\int_{-\ell_1(n)}^{\ell_2(n)}+\int_{\ell_2(n)}^{+\infty}\right)\left|\mathbb{P}\left(Y_{(n)}\leq a_{n}+b_{n}x\right)-e^{-e^{-x}}\right|dx\nonumber\\
	=&:\rm I+\rm {I\!I}+\rm {I\!I\!I},\nonumber
\end{align}
where $ \ell_1(n)$ and $ \ell_2(n) $ will be determined later with $ \ell_1(n)\to \infty $ and $ \ell_2(n)\to \infty $ as $ n\to \infty. $ For \eqref{total0}, it suffices to prove 
$$\aligned {\rm I\!I}&=2(1+o(1))\gamma_n^{-1}\log(\sqrt{2\pi}\log n)
\endaligned $$
and $${\rm I}\ll \gamma_n^{-1}\log\log n\quad\text{and}\quad {\rm I\!I\!I}\ll \gamma_n^{-1}\log\log n.$$  
 
Next, we are going to verify the estimates above one by one. 
Before that, we review the definition of $u_n(k, x)$ 
$$u_n(k, x)=\frac{k}{\sqrt{n}}+\sqrt{\gamma_n}+\frac{x}{\sqrt{\gamma_n}}$$ 
with $\gamma_n=\log n-2\log(\sqrt{2\pi}\log n).$
\subsection{Estimate on $\rm {I} $.}
Set $$ y_{0}=-\frac{a_{n}}{b_{n}}=-(\sqrt{n\gamma_{n}}+\gamma_{n}), $$ which is the unique solution to the equation $a_n+b_n x=0$ and then $\mathbb{P}(Y_{(n)}\le a_n+b_n x)=0$ for any $x<y_0.$ 
Therefore, 
$$ \aligned {\rm I}&=\int_{-\infty}^{-\ell_1(n)}\left|\mathbb{P}\left(Y_{(n)}\leq a_{n}+b_{n}x\right)-e^{-e^{-x}}\right|dx\\
&\leq\int_{y_{0}}^{-\ell_1(n)}\mathbb{P}\left(Y_{(n)}\leq a_{n}+b_{n}x\right)dx+\int_{-\infty}^{-\ell_1(n)}e^{-e^{-x}}dx. 
\endaligned $$
Using the substitution $y=e^{-x}$, we obtain 
\begin{equation}\label{C}
	\int_{-\infty}^{-\ell_1(n)}e^{-e^{-x}}dx=\int_{e^{\ell_1(n)}}^{+\infty}\frac{e^{-y}}{y}dy\leq \frac{1}{e^{\ell_1(n)}}\int_{e^{\ell_1(n)}}^{+\infty}e^{-y}dy=e^{-\ell_1(n)}e^{-e^{\ell_1(n)}} .
\end{equation}
After careful consideration, we choose $\ell_1(n)=\frac{1}{2}\log\log n,$ which satisfies $\ell_1(n)<\gamma_n.$  
Hence, 
\begin{align}\label{334}
	&\int_{y_{0}}^{-\ell_1(n)}\mathbb{P}\left(Y_{(n)}\leq a_{n}+b_{n}x\right)dx\\ \leq&\int_{y_{0}}^{-\gamma_{n}}\mathbb{P}\left(Y_{(n)}\leq a_{n}+b_{n}x\right)dx+\int_{-\gamma_{n}}^{-\ell_1(n)}\mathbb{P}\left(Y_{(n)}\leq a_{n}+b_{n}x\right)dx\nonumber\\ \leq& (-y_{0})\prod\limits_{k=0}^{m_{1}(n)}\mathbb{P}\left(Y_{n-k}\leq a_{n}-b_{n}\gamma_{n}\right)+\gamma_{n}\prod\limits_{k=0}^{m_{2}(n)}\mathbb{P}\left(Y_{n-k}\leq a_{n}-b_{n}\ell_1(n)\right)	
	\nonumber\\\leq&(-y_{0})\mathbb{P}^{m_{1}(n)}\left(Y_{n-m_{1}(n)}\leq a_{n}-b_{n}\gamma_{n}\right)+\gamma_{n}\exp\left\{-\sum\limits_{k=0}^{m_{2}(n)}\mathbb{P}\left(Y_{n-k}> a_{n}-b_{n}\ell_1(n)\right)\right\},\nonumber
\end{align}
where for the third and forth inequalities we use the property of $Y_{(n)}$ and the monotonicity of $\mathbb{P}(Y_k\le t)$ on $k$ to get $$\mathbb{P}(Y_{(n)}\le a_n+b_n x)=\prod_{k=1}^n\mathbb{P}(Y_{k}\le a_n+b_n x)\le \mathbb{P}^{m_1(n)}(Y_{n-m_1(n)}\le a_n-b_n \gamma_n)$$ for any $x\in (y_0, -\gamma_n)$ and the last inequality is also due to the inequality $ \log(1-t)\leq -t $
for $t\in (0, 1).$ 

On the one hand, choose $ m_{1}(n)=[\sqrt{n}] $ such that $u_{n}(m_{1}(n), -\gamma_{n})=1+O(n^{-1/2})$. The same argument as that for  Lemma \ref{ZX1}, we  know that
$$ \aligned \mathbb{P}\left(Y_{n-m_{1}(n)}> a_{n}-b_{n}\gamma_{n}\right)&=(1-\Phi(u_n(m_1(n), -\gamma_n)(1+O(m_1(n) n^{-1})))(1+o(1))\\
&=(1-\Phi(1+o(1)))(1+o(1))\geq \frac{1}{3},
\endaligned $$
which implies
$$(-y_{0})\mathbb{P}^{m_{1}(n)}\left(Y_{n-m_{1}(n)}\leq a_{n}-b_{n}\gamma_{n}\right)\leq (\sqrt{n\gamma_{n}}+\gamma_{n})\left(\frac{2}{3}\right)^{m_1(n)}\ll \gamma_n^{-1}  $$
On the other hand, 
\begin{align}
	u_{n}(0, -\ell_1(n))&=\sqrt{\gamma_{n}}\left(1+o(1)\right)\nonumber;\\
		u^{2}_{n}(0, -\ell_1(n))&=\gamma_{n}-2\ell_1(n)+o(1)\nonumber
\end{align}
and choose $m_{2}(n)$ such that$$u_{n}(m_{2}(n), -\ell_1(n))=2\sqrt{\log n},$$ whence Lemmas \ref{ZX1} and \ref{3.1} guarantee 
\begin{align}
	&\sum\limits_{k=0}^{m_{2}(n)}\mathbb{P}\left(Y_{n-k}> a_{n}-b_{n}\ell_1(n)\right)\nonumber\\
	=&\sum\limits_{k=0}^{m_{2}(n)}\frac{(1+o(1))}{\sqrt{2\pi}u_{n}(k, -\ell_1(n))}e^{-\frac{u_{n}^{2}(k,-\ell_1(n))}{2}}\nonumber\\
	=&\frac{\sqrt{n}(1+o(1))}{\sqrt{2\pi}u_{n}^{2}(0, -\ell_1(n))}e^{-\frac{u_{n}^{2}(0,-\ell_1(n))}{2}}-\frac{\sqrt{n}(1+o(1))}{\sqrt{2\pi}u_{n}^{2}(m_{2}(n), -\ell_1(n))}e^{-\frac{u_{n}^{2}(m_{2}(n),-\ell_1(n))}{2}}\nonumber\\
	=&\sqrt{\log n}\left(1+O\left(\frac{(\log\log n)^{2}}{\log n}\right)\right)+O\left(n^{-\frac{3}{2}}(\log n)^{-1}\right)\label{ss1}\\
	=&\sqrt{\log n}+o(1).\nonumber
\end{align}
Here, for the last but second equality we use the precise asymptotic 
$$e^{-\frac12 u_n^2(0, -\ell_1(n))}=e^{-\frac{\gamma_n}{2}+\ell_1(n)}(1+o(1))=O(n^{-1/2}(\log n)^{3/2})$$ and 
$$e^{-\frac12u_{n}^{2}(m_{2}(n),-\ell_1(n))}=e^{-2\log n}=n^{-2}.$$
The expression \eqref{ss1} implies that 
\begin{align}
	\gamma_{n}\exp\{-\sum\limits_{k=0}^{m_{2}(n)}\mathbb{P}\left(Y_{n-k}> a_{n}-b_{n}\ell_1(n)\right)\}&=\exp\left\{-\sqrt{\log n}+\log \gamma_{n}+o(1)\right\}\ll\gamma_n^{-1}\nonumber.
\end{align}
Combining all these estimates together, we conclude that
$$ \rm I\ll \gamma_n^{-1}. $$
 
\subsection{Estimate   on $\rm I\!I $.}
Setting $\beta_n(x)=-\sum_{k=1}^{n}\log \mathbb{P}(Y_{k}\le a_n+b_n x),$ it holds 
$${\rm I\!I}=\int_{-\ell_1(n)}^{\ell_2(n)}|e^{-\beta_n(x)}-e^{-e^{-x}}|dx.$$ Choosing $\ell_2(n)=\log(\sqrt{2\pi}\log n)$ with particular purpose one can see later, 	
we first analyze the asymptotic of $\beta_n(x)$ for $x\in (-\ell_1(n), \ell_2(n)).$ By definition and the fact that $-\log \mathbb{P}(Y_k\le a_n+b_n x)$ is increasing on $k,$ with some $1\ll j_n\ll n,$ we have 
\begin{equation}\label{betanxlu}\aligned \beta_n(x)&\le -\sum_{k=0}^{j_n-1} \log \mathbb{P}(Y_{n-k}\le a_n+b_n x)-(n-j_n)\log \mathbb{P}(Y_{n-j_n}\le a_n+b_n x);\\
\beta_n(x)&\ge -\sum_{k=0}^{j_n-1} \log \mathbb{P}(Y_{n-k}\le a_n+b_n x). \endaligned \end{equation}
Now $1\ll u_n(k, x)$ and then Lemma \ref{ZX1} ensures $\mathbb{P}(Y_{n-k}\ge a_n+b_n x)=o(1)$ 
uniformly on $0\le k\le j_n$ and $x\ge -\ell_1(n).$ Choose $j_n=[n^{3/5}]$  to make sure $u_n(j_n, x)=n^{\frac{1}{10}}+O(\sqrt{\gamma_n})$  such that    
\begin{equation}\label{betanx2}\aligned -n\log \mathbb{P}(Y_{n-j_n}\le a_n+b_n x)&=n\mathbb{P}(Y_{n-j_n}\ge a_n+b_n x)(1+o(1))\\
&=\frac{n(1+o(1))}{\sqrt{2\pi}u_{n}(j_n, x)}e^{-\frac{u_n^2(j_n, x)}{2}}\\
&=o(n^{\frac{9}{10}} e^{-\frac13 n^{\frac15}}).\\
 \endaligned 
 \end{equation}
Meanwhile, since $u_n(k, x)=\frac{k}{\sqrt{n}}+\sqrt{\gamma_n}+o(1)$ uniformly on $x\in (-\ell_1(n), \ell_2(n)),$ it follows
$$\aligned u_n^{-2}(k, x)=\tilde{O}(\gamma_n^{-1})
\endaligned$$
for $0\le k\le j_n.$ Thus, 
 Lemma \ref{ZX1} says
$$\mathbb{P}(Y_{n-k}\ge a_n+b_n x)=\frac{1+\tilde{O}(\gamma_{n}^{-1})}{\sqrt{2\pi}u_n(k, x)} e^{-\frac{u_n^2(k, x)}{2}}.$$ 
Therefore,  on the one hand 
$$\mathbb{P}(Y_{n-k}\ge a_n+b_n x)\le \mathbb{P}(Y_n\ge a_n+b_n x) =O(u_n^{-1}(0, x)e^{-\frac{u_n^2(0, x)}{2}})$$ 
and on the other hand Lemma \ref{3.1} indicates  
$$\aligned &\quad \sum_{k=0}^{j_n-1} \mathbb{P}(Y_{n-k}\ge a_n+b_n x)=\sum_{k=0}^{j_n-1}\frac{1+\tilde{O}(\gamma_{n}^{-1})}{\sqrt{2\pi} u_n(k, x)} e^{-\frac{u_n^2(k, x)}{2}}=\frac{(1+O(\gamma_n^{-1}))\sqrt{n}e^{-\frac{u_n^2(0, x)}{2}}}{\sqrt{2\pi} u_n^2(0, x)},
\endaligned$$
where for the last equality we use Lemma \ref{3.1} and the fact 
$$\frac{1}{u_n^2(j_n, x)} e^{-\frac{u_n^2(j_n, x)}{2}}\ll \frac{1}{\gamma_nu_n^2(0, x)} e^{-\frac{u_n^2(0, x)}{2}}.$$

Now $u_n(0, x)=\sqrt{\gamma_n}+\frac{x}{\sqrt{\gamma_n}}$ and then 
$$e^{-\frac{u_n^{2}(0, x)}{2}}=e^{-\frac{1}{2}\gamma_n-x-\frac{x^2}{2\gamma_n}}=\frac{\sqrt{2\pi}\log n}{\sqrt{n}}e^{-x-\frac{x^2}{2\gamma_n}},$$

which implies $\mathbb{P}(Y_n\ge a_n+b_n x)\ll\gamma_n^{-2}$ and 

$$
\sum_{k=0}^{j_n-1} \mathbb{P}(Y_{n-k}\ge a_n+b_n x)=\frac{(1+O(\gamma_n^{-1}))\log n}{\gamma_n(1+\frac{x}{\gamma_n})^2}e^{-x-\frac{x^2}{2\gamma_n}}.	
$$
Thereby,
\begin{equation}\label{logjnpart}\aligned
-\sum_{k=0}^{j_n-1}\log\mathbb{P}(Y_{n-k}\le a_n+b_n x)&=\sum_{k=0}^{j_n-1} \mathbb{P}(Y_{n-k}\ge a_n+b_n x)(1+O(\mathbb{P}(Y_{n-k}\ge a_n+b_n x)))	\\
&=\frac{(1+O(\gamma_n^{-1}))\log n}{\gamma_n(1+\frac{x}{\gamma_n})^2}e^{-x-\frac{x^2}{2\gamma_n}}.
\endaligned
\end{equation}
Putting \eqref{betanx2} and \eqref{logjnpart} back into \eqref{betanxlu}, 
we see 
\begin{equation}\label{keybetanx}
\beta_n(x)=\frac{(1+O(\gamma_n^{-1}))\log n}{\gamma_n(1+\frac{x}{\gamma_n})^2}e^{-x-\frac{x^2}{2\gamma_n}}.	
\end{equation}
and then 
$$
\aligned e^{-x}-\beta_n(x)&=e^{-x}\big(1-e^{-\frac{x^2}{2\gamma_n}}\frac{(1+O(\gamma_n^{-1}))\log n}{\gamma_n(1+\frac{x}{\gamma_n})^2}\big)	\\
&=e^{-x}\gamma_n^{-1}(1+\gamma_n^{-1}x)^{-2}\left(\gamma_n(1+\gamma_n^{-1}x)^2-e^{-\frac{x^2}{2\gamma_n}}\log n(1+O(\gamma_n^{-1}))\right).
\endaligned 
$$
Since the condition $\ell_1(n)+\ell_2(n)\ll \sqrt{\gamma_n}$  guarantees $x^2\ll \gamma_n,$ it follows from some simple calculus together with $e^{-t}=1-t+O(t^2)$ for $|t|$ small enough that 
$$ \aligned 
&\quad \gamma_n(1+\gamma_n^{-1}x)^2-e^{-\frac{x^2}{2\gamma_n}}(1+O(\gamma_n^{-1}))\log n\\
&=\gamma_n+2x+\frac{x^2}{\gamma_n}-(1-\frac{x^2}{2\gamma_n}+O(x^4\gamma_n^{-2}))(1+O(\gamma_n^{-1}))\log n\\
&=\gamma_n-\log n+2x+\frac {x^2}2+O(1)\\
&=(-2\ell_2(n)+2x+\frac {x^2}2)(1+o(1)).
\endaligned $$
As a consequence, it holds 
\begin{equation}\label{betaxexpre}
\aligned e^{-x}-\beta_n(x)&=e^{-x}\gamma_n^{-1}(1+o(1))(2(x-\ell_2(n))+\frac{x^2}{2}).
\endaligned 
\end{equation} 
The choices of $\ell_1(n)=\frac{1}{2}\log\log n$ and $\ell_2(n)=\log(\sqrt{2\pi}\log n)$ claim
$$e^{-x}\gamma_n^{-1}\big|2(x-\ell_2(n))+\frac{x^2}{2}\big|=\tilde{O} (e^{\ell_1(n)} \gamma_n^{-1}(\log\log n)^2)=o(1),$$
which is exactly the reason why both $\ell_1(n)$ and $\ell_2(n)$ are chosen to be of order $\log\log n.$

The expression \eqref{betaxexpre} in further entails 
\begin{equation}\label{IIfinal}\aligned {\rm I\!I}
&=\int_{-\ell_1(n)}^{\ell_2(n)}e^{-e^{-x}}|e^{e^{-x}-\beta_n(x)}-1|dx
\\
&=(1+o(1))\gamma_n^{-1}\int_{-\ell_1(n)}^{\ell_2(n)}e^{-e^{-x}}e^{-x}\big|2(x-\ell_2(n))+\frac{x^2}{2}\big|dx.  \quad
\endaligned  \end{equation}
Now  
$$\aligned \int_{-\infty}^{+\infty}e^{-e^{-x}}e^{-x} dx &=\int_0^{+\infty} e^{-t} dt=1; \\
\int_{-\infty}^{+\infty}e^{-e^{-x}}e^{-x} x^2dx&=\int_0^{+\infty} e^{-t}(\log t)^2 dt=\gamma^2+\frac{\pi^2}{6},\\
\endaligned $$
where $\gamma$ is the Euler constant and the last integral can be found in \cite{GRJ} and $$\int_{-\infty}^{+\infty}e^{-e^{-x}}e^{-x} |x| dx=\int_0^{+\infty} e^{-t}|\log t|\, dt$$ converges,  
one gets since $\ell_1(n), \ell_2(n)\to\infty$ that 
\begin{equation}\label{IIfinal2}{\rm I\!I}=(1+o(1))\gamma_n^{-1}\left(2\ell_2(n)+O(1)\right)=2(1+o(1))\gamma_n^{-1}\ell_2(n)=\frac{2(1+o(1))\log\log n}{\log n}.\end{equation}

\subsection{Estimate on {\rm I\!I\!I}} Recall $$\beta_n(x)=-\sum_{j=1}^n\log \mathbb{P}(Y_j\le a_n+b_n x)$$ and then 
$${\rm I\!I\!I}=\int_{\ell_2(n)}^{+\infty} |e^{-\beta_n(x)}-e^{-e^{-x}}| dx,$$
whence it follows from the elementary inequality $1-e^{-x}\le x$ that 
$$\aligned{\rm I\!I\!I}&\le \int_{\ell_2(n)}^{+\infty} (e^{-x}+\beta_n(x))dx=\frac{1}{\sqrt{2\pi}\log n}+\int_{\ell_2(n)}^{+\infty} \beta_n(x)dx. 
 \endaligned $$ 
 As mentioned above, $\mathbb{P}(Y_{n-k}\ge a_n+b_n x)=o(1)$ for $x\ge \ell_2(n)$ and $0\le k\le j_n.$ Thus, to prove 
 $$\int_{\ell_2(n)}^{+\infty} \beta_n(x)dx\ll\frac{\log\log n}{\gamma_n},$$ which implies immediately $${\rm I\!I\!I}\ll \frac{\log\log n}{\gamma_n},$$ it suffices to prove 
  \begin{equation}\label{iii} J_n:=\int_{\ell_2(n)}^{+\infty} \sum_{j=1}^{n}\mathbb{P}(Y_{j}\ge a_n+b_n x)dx\ll\frac{\log\log n}{\gamma_n}. 
 \end{equation}
 
 Setting
$t_{n}=[ n^{\frac56}], $ and $j_{n}=[ n^{\frac{3}{5}}]$ and $p_n=n^{\frac{1}{10}}\sqrt{\gamma_n},$
similarly using the monotonicity
of $\mathbb{P}(Y_{j} \geq a_n+b_n x)$ on $j,$ we see
\begin{equation}\label{kkkkk}
	J_n\leq n\int_{\ell_2(n)}^{+\infty}\mathbb{P}(Y_{n-t_{n}}>a_{n}+b_{n}x)dx +\int_{\ell_2(n)}^{+\infty}\sum\limits_{k=0}^{t_{n}}\mathbb{P}(Y_{n-k}>a_{n}+b_{n}x)dx.
\end{equation}

Then, it follows from Lemmas  \ref{ZX1} and \ref{3.1}, and the fact  $ u_{n}(t_{n},\ell_2(n))=n^{1/3}(1+o(1)) $ that 
\begin{equation}\label{31&}\aligned
	n\int_{\ell_2(n)}^{+\infty}\mathbb{P}(Y_{n-t_{n}}>a_{n}+b_{n}x)dx\leq& \int_{\ell_2(n)}^{+\infty}\frac{n}{u_{n}(t_{n},x)}e^{-\frac{u_{n}^{2}(t_{n},x)}{2}} +O(n^{\frac12}u_{n}^{-3}(t_{n},x))dx\\
=&O\left(\frac{n\sqrt{\gamma_n}}{u_{n}^{2}(t_{n},\ell_2(n))}e^{-\frac{u_{n}^{2}(t_{n},\ell_2(n))}{2}}+\frac{\sqrt{n\gamma_{n}}}{u_{n}^{2}(t_{n},\ell_2(n))}\right)	\\
=	&O\left(\frac{\sqrt{\gamma_{n}}}{n^{\frac{1}{6}}}\right)
\ll\gamma_{n}^{-1}\log\log n.\endaligned 
\end{equation}  

We further refine the second integral in \eqref{kkkkk} as 
\begin{equation}\label{31*}\aligned
	&\quad \int_{\ell_2(n)}^{+\infty}\sum\limits_{k=0}^{t_{n}}\mathbb{P}(Y_{n-k}>a_{n}+b_{n}x)dx\\
	\le&\int_{\ell_2(n)}^{p_{n}}\sum\limits_{k=0}^{j_{n}-1}\mathbb{P}(Y_{n-k}>a_{n}+b_{n}x)dx+\int_{\ell_2(n)}^{p_{n}}\sum\limits_{k=j_{n}}^{t_{n}}\mathbb{P}(Y_{n-k}>a_{n}+b_{n}x)dx\\
	&+\int_{p_{n}}^{+\infty}\sum_{k=0}^{t_n}\mathbb{P}(Y_{n-k}>a_{n}+b_{n}x)dx.\endaligned 
\end{equation}
Applying Lemmas \ref{ZX1} and \ref{3.1}, we get
\begin{equation}\label{310}\aligned
	\int_{\ell_2(n)}^{p_{n}}\sum\limits_{k=0}^{j_{n}-1}\mathbb{P}(Y_{n-k}>a_{n}+b_{n}x)dx
	=&\int_{\ell_2(n)}^{p_{n}}\sum\limits_{k=0}^{j_{n}-1}\frac{1+o(1)}{\sqrt{2\pi}u_{n}(k,x)}e^{-\frac{u_{n}^{2}(k,x)}{2}}dx\\
	\leq&\int_{\ell_2(n)}^{+\infty}\sum\limits_{k=0}^{+\infty}\frac{1}{u_{n}(k,x)}e^{-\frac{u_{n}^{2}(k,x)}{2}}dx\\
	=&O\left(\frac{\sqrt{n\gamma_{n}}}{u_{n}^{3}(0,\ell_2(n))}e^{-\frac{u_{n}^{2}(0,\ell_2(n))}{2}}\right)\\
	=&O(\gamma_n^{-1})
	\\
	\ll&\gamma_{n}^{-1}\log\log n,
	\endaligned 
\end{equation}
where we use the fact $u_{n}(0,\ell_2(n))=\sqrt{\gamma_{n}}+\frac{\ell_2(n)}{\sqrt{\gamma_{n}}}$ and then  
$$e^{-\frac{u_{n}^{2}(0,\ell_2(n))}{2}}=e^{-\frac12\gamma_n-\log(\sqrt{2\pi}\log n)}(1+o(1))=\frac{1}{\sqrt{n}}(1+o(1)).$$
Next, we work on the second integral in \eqref{31*}. Note that $$ u_{n}(j_{n},\ell_2(n))=\frac{j_{n}}{\sqrt{n}}+\sqrt{\gamma_{n}}+\frac{\ell_2(n)}{\sqrt{\gamma_{n}}}.$$  Using the same method as for (\ref{310}), we obtain 
\begin{equation}\label{311}\aligned
	\int_{\ell_2(n)}^{p_{n}}\sum\limits_{k=j_{n}}^{t_{n}}\mathbb{P}(Y_{n-k}>a_{n}+b_{n}x)dx
	\leq&\int_{\ell_2(n)}^{p_{n}}\sum\limits_{k=j_{n}}^{t_{n}}u_{n}^{-1}(k,x)e^{-\frac{u_{n}^{2}(k,x)}{2}}+O(n^{-\frac{1}{2}}u_{n}^{-3}(k,x))dx\\
	=&O\left(\frac{\sqrt{n\gamma_{n}}}{u_{n}^{3}(j_{n},\ell_2(n))}e^{-\frac{u_{n}^{2}(j_{n},\ell_2(n))}{2}}+\frac{\sqrt{\gamma_{n}}}{u_{n}(j_{n},\ell_2(n))}\right)\\
	=&O(\sqrt{\gamma_{n}}n^{-\frac{1}{10}})\ll\gamma_{n}^{-1}	. \endaligned 
\end{equation}
At last, Lemmas \ref{ZX1} and \ref{3.1} work together to indicate  
\begin{equation}\label{312}\aligned
	\int_{p_{n}}^{+\infty}\sum\limits_{k=0}^{t_{n}}\mathbb{P}(Y_{n-k}>a_{n}+b_{n}x)dx&\leq \int_{p_{n}}^{+\infty}\sum\limits_{k=0}^{t_{n}}(u_{n}^{-1}(k,x)e^{-\frac{u^{2}_{n}(k,x)}{2}}+O(n^{-1/2}u_{n}^{-3}(k,x))dx\\
	&=O\left(\frac{\sqrt{n\gamma_{n}}}{u^{3}_{n}(0,p_{n})}e^{-\frac{u^{2}_{n}(0,p_{n})}{2}}+\frac{\sqrt{\gamma_{n}}}{u_{n}(0,p_{n})}\right)\\
	&=O(n^{-\frac{1}{10}})\ll\gamma_{n}^{-1}. \endaligned 
\end{equation}
Inserting \eqref{310}, \eqref{311} and \eqref{312}  into \eqref{31*}, we know 
$$\int_{\ell_2(n)}^{+\infty}\sum\limits_{k=0}^{t_{n}}\mathbb{P}(Y_{n-k}>a_{n}+b_{n}x)dx\ll \gamma_n^{-1}\log\log n,	
$$
which, combining with \eqref{kkkkk} and \eqref{31&}, confirms \eqref{iii}.
 
\subsection{Proof of \eqref{total1}} 
For $W_1(\mathcal{L}(W_n), F_n),$ we still use the closed formula of $L1$-Wasserstein distance to write 
\begin{align}
	W_{1}(\mathcal{L}(W_n), F_{n}) &=\int_{-\infty}^{+\infty}\left|\mathbb{P}\left(\sqrt{4\gamma_{n}}\left(R_{n}-\sqrt{n}-\sqrt{\frac{\gamma_{n}}{4}}\right)\leq x\right)-\mathbb{P}\left(Y_{(n)}\leq a_{n}+b_{n}x\right)\right|dx\nonumber\\
	&=\int_{-\infty}^{+\infty}\left|\mathbb{P}\big(\sqrt{Y_{(n)}}\leq \frac{x}{\sqrt{4\gamma_{n}}}+\sqrt{n}+\sqrt{\frac{\gamma_{n}}{4}}\big)-\mathbb{P}\left(Y_{(n)}\leq a_{n}+b_{n}x\right)\right|dx.\nonumber
\end{align}
Note that 
\begin{align}
	\left(\frac{x}{\sqrt{4\gamma_{n}}}+\sqrt{n}+\sqrt{\frac{\gamma_{n}}{4}}\right)^{2}&=n+\left(\frac{x}{\sqrt{4\gamma_{n}}}+\sqrt{\frac{\gamma_{n}}{4}}\right)^{2}+2\sqrt{n}\left(\frac{x}{\sqrt{4\gamma_{n}}}+\sqrt{\frac{\gamma_{n}}{4}}\right)\nonumber\\
	&=n+\sqrt{n\gamma_{n}}+\sqrt{\frac{n}{\gamma_{n}}}x+\left(\frac{x}{\sqrt{4\gamma_{n}}}+\sqrt{\frac{\gamma_{n}}{4}}\right)^{2}\nonumber\\
	&\geq a_{n}+b_{n}x\nonumber,
\end{align}
which helps to write 
$$\aligned 
&\quad W_{1}(\mathcal{L}(W_n), F_{n}) \\
&=\int_{-\infty}^{+\infty}\bigg(\mathbb{P}\big(\sqrt{Y_{(n)}}\le \frac{x}{\sqrt{4\gamma_{n}}}+\sqrt{n}+\sqrt{\frac{\gamma_{n}}{4}}\bigg)-\mathbb{P}\left(Y_{(n)}\leq a_{n}+b_{n}x\right)\bigg)dx\\
	&=\int_{y_1}^{+\infty}\mathbb{P}\big(\sqrt{Y_{(n)}}\leq\frac{x}{\sqrt{4\gamma_{n}}}+\sqrt{n}+\sqrt{\frac{\gamma_{n}}{4}}\big) dx-\int_{y_0}^{\infty}\mathbb{P}\left(Y_{(n)}\leq a_{n}+b_{n}x\right)dx,
\endaligned
$$
where $ y_{1}:=-\sqrt{4\gamma_{n}}\left(\sqrt{n}+\sqrt{\frac{\gamma_{n}}{4}}\right)$ is the unique solution to $\frac{x}{\sqrt{4\gamma_{n}}}+\sqrt{n}+\sqrt{\frac{\gamma_{n}}{4}}=0.$ 
Using the substitution $t=\frac{x}{\sqrt{4\gamma_{n}}}+\sqrt{n}+\sqrt{\frac{\gamma_{n}}{4}}$ and $t=a_n+b_n x,$ respectively, we see 
$$\aligned 
W_{1}(\mathcal{L}(W_n), F_{n})
&=\sqrt{4\gamma_n}\int_{0}^{+\infty} \mathbb{P}\big(\sqrt{Y_{(n)}}\le t \big)d t-\sqrt{\frac{\gamma_n}{n}}\int_0^{+\infty}\mathbb{P}(Y_{(n)}\leq t) \,d t\\
	&=\sqrt{\gamma_n}\int_0^{\infty} (\frac{1}{\sqrt{t}}-\frac{1}{\sqrt{n}}) \mathbb{P}(Y_{(n)}\leq t) dt\\
		&\le \sqrt{\gamma_n}\int_0^{n} (\frac{1}{\sqrt{t}}-\frac{1}{\sqrt{n}}) \mathbb{P}(Y_{(n)}\leq t) dt\\
		&\le \sqrt{\gamma_n n}\,\mathbb{P}(Y_{(n)}\le n)\\
		&\le \sqrt{\gamma_n n} \;\mathbb{P}^{[\sqrt{n}]}(Y_{n-[\sqrt{n}]}\le n).
		\endaligned 
$$
Observe $n=a_n-b_n \gamma_n,$ which satisfies $u_n([\sqrt{n}], -\gamma_n)=1+O(n^{-1/2}),$ and then Lemma \ref{ZX1} tells 
 $$ \mathbb{P}(Y_{n-[\sqrt{n}]}>n)=(1-\Phi(1+O(n^{-\frac{1}{2}}))(1+o(1))\geq \frac{1}{3}, $$  
which implies
$$ 2\sqrt{n\gamma_{n}}\mathbb{P}^{[\sqrt{n}]}(Y_{n-[\sqrt{n}]}\leq n)\leq 2\sqrt{n\gamma_{n}}\left(\frac{2}{3}\right)^{[\sqrt{n}]}.$$ 
Therefore,
\begin{equation}\label{wwwww}
	W_{1}(\mathcal{L}(W_n), F_{n})\ll \gamma_n^{-1}.
\end{equation}
The proof of Theorem \ref{TH} is completed now.

\section{Proof of Theorem \ref{th2}}
This section is devoted to the proof of Theorem \ref{th2}. 
Let $F_n$ and $\bar{F}_n$ be the distribution functions of $X_n$ and $W_n,$ respectively and for Theorem \ref{th2}, it suffices to prove 
\begin{equation}\label{789}
	\lim\limits_{n\to\infty}\frac{\log n}{\log\log n}\sup\limits_{x\in\mathbb{R}}|\bar{F}_n(x)-e^{-e^{-x}}|=\frac{2}{e}
\end{equation}
and
\begin{equation}\label{8900}
	\sup\limits_{x\in  \mathbb{R}}|F_{n}(x)-\bar{F}_n(x)|\ll \gamma_{n}^{-1}\log \log n. 
\end{equation}
Now \eqref{8900} is a natural result of $W_1(\mathcal{L}(W_n), F_n)\ll \gamma_n^{-1}$ and we only need to verify \eqref{789}.
Examining the proof for $W_1$ distance, one sees 
\begin{equation}\label{233}\aligned
	\sup\limits_{x\in\mathbb{R}}|\bar{F}_n(x)-e^{-e^{-x}}|=\sup_{x\in (-\ell_1(n), \ell_2(n))}|e^{-\beta_n(x)}-e^{-e^{-x}}|
	\endaligned 
\end{equation}
and then it follows from 
\eqref{betaxexpre} that  
$$\sup\limits_{x\in(-\ell_1(n),\ell_{2}(n))}|e^{-\beta_{n}(x)}-e^{-e^{-x}}|=\gamma_n^{-1}(1+o(1))\sup_{x\in(-\ell_1(n),\ell_{2}(n))}e^{-e^{-x}-x}\big|2(\ell_2(n)-x)-\frac{x^2}{2}\big|.$$
The simple inequality $$2(\ell_2(n)-x)-\frac{x^2}{2}\le \big|2(\ell_2(n)-x)-\frac{x^2}{2}\big|\le 2(\ell_2(n)-x)+\frac{x^2}{2}$$
implies 
$$\aligned &\quad\sup_{x\in(-\ell_1(n),\ell_{2}(n))}e^{-e^{-x}-x}\big|2(\ell_2(n)-x)-\frac{x^2}{2}\big|\\
&\le 2 \sup_{x\in(-\ell_1(n),\ell_{2}(n))}e^{-e^{-x}-x}(\ell_2(n)-x)+\frac12\sup_{x\in(-\ell_1(n),\ell_{2}(n))}e^{-e^{-x}-x}x^2\endaligned $$
and similarly 
$$\aligned &\quad\sup_{x\in(-\ell_1(n),\ell_{2}(n))}e^{-e^{-x}-x}\big|2(\ell_2(n)-x)-\frac{x^2}{2}\big|\\
&\ge 2\sup_{x\in(-\ell_1(n),\ell_{2}(n))}e^{-e^{-x}-x}(\ell_2(n)-x)-\frac12\sup_{x\in(-\ell_1(n),\ell_{2}(n))}e^{-e^{-x}-x}x^2.\endaligned $$ 

Set $$g(x)=e^{-x}+x-\log  (\ell_2(n)-x), \quad -\ell_1(n)\le x\le \ell_2(n)$$
and we are going to identify the minimizer of $g.$  
Now $$g'(x)=1-e^{-x}+\frac{1}{\ell_2(n)-x},$$ which is strictly increasing and $g'(-\ell_1(n))<0$ and $g'(\ell_2(n))=+\infty.$ Thus, $g$ has a unique minimizer, denoted by $z_0.$ 
It is ready to know $g'(0)>0$ and 
$$\aligned g'(-\frac{1}{\ell_2(n)})&=1-e^{\frac{1}{\ell_2(n)}}+\frac{1}{\ell_2(n)+\frac{1}{\ell_2(n)}}\\
&=-\frac{1}{\ell_2(n)}-\frac{1}{2\ell_2^2(n)}+\frac{1}{\ell_2(n)+\frac{1}{\ell_2(n)}}+o(\ell_2^{-2}(n))\\
&<0,
\endaligned $$ 
whence, $$-\frac{1}{\ell_2(n)}<z_0<0.$$ 
As a consequence 
$$g(z_0)<g(0)=1-\log \ell_2(n)$$ and also it follows from the monotonicity of involved functions that 
$$g(z_0)>1-\frac{1}{\ell_2(n)}-\log(\ell_2(n)+\ell_2^{-1}(n))=1-\log\ell_2(n)+o(1).$$
Therefore, 
$$\aligned \sup_{x\in (-\ell_1(n), \ell_2(n))}e^{-e^{-x}-x}(\ell_2(n)-x)=e^{-g(z_0)}=e^{-1} \ell_2(n)(1+o(1)).
\endaligned $$
Also, it is easy to check that 
 $$\sup_{x\in(-\ell_1(n),\ell_{2}(n))}e^{-e^{-x}-x}x^2=O(1).$$
Therefore,
$$ \sup\limits_{x\in\mathbb{R}}|\bar{F}_n(x)-e^{-e^{-x}}|=\frac{2\ell_2(n)}{e\gamma_{n}}(1+o(1)), $$
which completes the proof of \eqref{789}.

\subsection*{Conflict of interest}  The authors have no conflicts to disclose.

\subsection*{Acknowledgment} The authors would like to thank Professor Xinxin Chen from Beijing Normal University for help discussions on topics related to this work.

\end{document}